\documentclass[12pt]{article}
\usepackage{amssymb, amsmath, amsthm, amsfonts}
\usepackage[italian, english]{babel}
\usepackage{graphicx}
\usepackage[T1]{fontenc}
\usepackage[latin1]{inputenc}
\usepackage{times}

\newtheorem{thm}{Theorem}[section]
\newtheorem{lemma}[thm]{Lemma}
\newtheorem{defn}[thm]{Definition}

\newtheorem{proposition}[thm]{Proposition}
\newtheorem{remark}[thm]{Remark}
\newtheorem{example}[thm]{Example}

    \newcommand{\ct}[1]{\langle {#1}\rangle \lower.3ex\hbox{$_{t}$}}
    \newcommand{\lt}[1]{[ {#1}] \lower.3ex\hbox{$_{t}$}}

\newcommand{\Part}{{\cal P}}
\newcommand{\R}{{\mathbb{R}}}
\newcommand{\N}{{\mathbb N}}

\newcommand{\C}{{\mathcal C}}

\newcommand{\K}{{\mathcal K}}
\newcommand{\sfe}{{\mathbb S}^{n-1}}

\newcommand{\Gr}{{\rm Gr}}

\newcommand{\cl}{{\rm cl}}
\newcommand{\interno}{{\rm int}}

\newcommand{\supp}{{\rm supp}}
\newcommand{\Val}{\mbox{\rm Val}}
\newcommand{\Vol}{{\rm Vol}}
\newcommand{\KL}{{\rm Kl}}
\newcommand{\rot}{{\bf O}}
\newcommand{\V}{{\rm V}}

\begin{document}
\selectlanguage{english}
\title{Translation invariant valuations 
\\
on quasi-concave functions}
\author{Andrea Colesanti, Nico Lombardi, Lukas Parapatits}
\date{}
\maketitle

\begin{abstract}
We study real-valued, continuous and translation invariant valuations defined on the space of quasi-concave functions of $N$ variables. In particular, we
prove a homogeneous decomposition theorem of McMullen type, and we find a representation formula for those
valuations which are $N$-homogeneous. Moreover, we introduce the notion of Klain's functions for these type of valuations.

\medskip

\end{abstract}

\section{Introduction}

In this paper we study real-valued, continuous and translation invariant valuations on the space of quasi-concave functions of $N$ variables.
In particular, we prove a homogeneous decomposition theorem, and we give a representation formula for those which are additionally 
$N$-homogeneous. 

\medskip

By a quasi-concave function it is generally meant a function having convex super-level sets. To be more specific,
in our notation $f\colon\R^N\to [0,\infty)$, $N\ge1$, is quasi-concave if, for every $t>0$, 
$$
L_t(f)=\{x\colon f(x)\ge t\}
$$
is, while not empty, a compact convex set (i.e.\ a convex body). We denote by $\C^N$ the space of functions with this property. 

A valuation on the space $\C^N$ is a functional $\mu\colon\C^N\to\R$ such that
$$
\mu(f\vee g)+\mu(f\wedge g)=\mu(f)+\mu(g)
$$
for every $f,g\in\C^N$ such that $f\vee g\in\C^N$ (note that $f\wedge g\in\C^n$ whenever $f,g\in\C^n$). Here $\vee$ and $\wedge$ denote the 
point-wise maximum and minimum operations, respectively. 

\medskip

The study of valuations on function spaces began quite recently, as a natural development of the theory of valuations defined on the space $\K^N$
of convex bodies of $\R^N$. We recall that a (real-valued) valuation on the space $\K^N$is a 
functional $\mu\colon\K^N\to\R$ such that the following finite additivity property is fulfilled: 
$$
\mu(K\cup L)+\mu(K\cap L)=\mu(K)+\mu(L),
$$
for every $K,L\in\K^N$ such that $K\cup L\in\K^N$. Hence, passing from sets to functions, in the definition of valuation the operations
of union and intersection are replaced by those of maximum and minimum. The theory of valuations on $\K^N$ is currently one of the
most active branches of convex geometry. 
A great impulse to the research in this area was given by deep and fundamental results like the Hadwiger classification theorem of rigid motion 
invariant valuations, and the McMullen decomposition theorem for translation invariant valuations. These breakthroughs motivated
many interesting developments and a finer and finer description of valuations on $\K^N$, subject to different invariance conditions. 
An updated and exhaustive survey on the state of the art in this area is contained in \cite[Chapter 6]{Schneider}. 

\medskip

Coming back to valuations on functions spaces, in this context as well one typically aims to give a complete classifications of valuations
defined on a given space of functions, having some invariance properties. The results achieved so far embrace various types 
of spaces, such as Lebesgue spaces, Sobolev spaces, functions of bounded variations, convex functions, etc. For more details 
the reader is referred to the papers: [\cite{BGW}-\cite{Colesanti-Lombardi},
\cite{Kone}-\cite{Ma}, \cite{Ober}, 
\cite{Tsang-2010}-\cite{Wright}]. 

\medskip

In the paper \cite{Colesanti-Lombardi}, the first two authors considered valuations $\mu$ on $C^N$ which are invariant with respect to the composition 
with rigid motions, i.e.
$$
\mu(f\circ T)=\mu(f)
$$
for every $f\in C^N$ and for every isometry $T$ (i.e. the composition of a rotation and a translation) of $\R^N$. Moreover, the continuity of $\mu$
was assumed with respect to a topology that will be used in the present paper as well and that we briefly recall. Let 
$f_i$, $i\in\N$, be a sequence in $\C^N$; we say that $f_i$ converges to $f\in\C^N$ if
$$
\lim_{i\to\infty}f_i(x)=f(x),\quad\forall\, x\in\R^N,
$$
and the sequence $f$ is either point-wise decreasing or increasing. A valuation on $\C^N$ is called continuous if it is continuous 
with respect to this type of convergence. 

Rigid motion invariant and continuous valuations on $\C^N$ were characterized in \cite{Colesanti-Lombardi} (see also 
Theorem \ref{characterization} of the present paper) by a result which closely resembles Hadwiger's characterization theorem 
for rigid motion invariant and continuous valuations on $\K^N$, and in fact its proof is deeply indebted to this theorem.

\medskip

Here we turn our attention to continuous valuations on $\C^N$ which are translation invariant only. In the corresponding situation for 
valuations on $K^N$, several important results are known. The celebrated McMullen decomposition theorem (see
\cite{McMullen}) asserts that any continuous and translation invariant valuation on $K^N$ can be written as the sum of $N+1$ homogeneous valuations,
with integer order of homogeneity (between $0$ and $N$). We find a counterpart of this theorem in $\C^N$. A valuation $\mu\colon \C^N\to\R$ is said to be
homogeneous of order $k\in\R$ if 
$$
\mu(f_\lambda)=\lambda^k\mu(f),\quad\forall\, f\in\C^N,\,\forall\,\lambda>0,
$$
where $f_\lambda$ is defined as follows
$$
f_\lambda(x)=f\left(\frac x\lambda\right),\quad x\in\R^N.
$$

\begin{thm}\label{teorema 1}
Let $\mu\colon{\cal C}^N\rightarrow \mathbb{R}$ be a continuous and translation invariant valuation.
Then, for every $k\in\{0,\dots,N\}$ there exist a continuous, translation invariant and $k$-homogeneous valuation $\mu_{k}$ such that
$$
\mu=\sum_{k=0}^{N}\mu_{k}.
$$
Moreover, the valuations $\mu_{k}$ are univocally determined.
\end{thm}

As proved by Hadwiger in \cite{Hadwiger} (see also Theorem 6.4.3 in \cite{Schneider}), among homogeneous valuations on $\K^N$, those with order of 
homogeneity $N$ are of a very simple form, as each of them is just a multiple of the ordinary volume $V_N$ (the Lebesgue measure). Here we
also provide a more specific description of $N$-homogeneous valuations on $\C^N$. Let $f\in\C^N$; for every level $t>0$ we may consider the volume
of the corresponding super-level set of $f$:
$$
u_N(t):=V_N(\{x\colon f(x)\ge t\}).
$$
The function $u_N$ is decreasing; let us denote by $S_N(f;\cdot)$ the non-negative measure representing the distributional derivative of $-u_N$. We establish 
the following result.

\begin{thm}\label{teorema 2}
Let $\mu$ be a continuous, translation invariant and $N$-homogeneous valuation on $\C^N$. There exists a unique function $c\colon[0,+\infty)\rightarrow \R$ such that 
\begin{equation}\label{intro8}
\mu(f)=\int_{(0,+\infty)}c(t)dS_{N}(f;t)
\end{equation}
for all $f\in {\cal C}^N$. Moreover, $c$ is continuous, and there exists $\delta>0$ such that $c(t)=0$, for all $t\in [0,\delta]$.
\end{thm}

As explained in Remark \ref{monotone case}, 
the representation formula \eqref{intro8} becomes more readable when $\mu$ is also assumed to be monotone, e.g. increasing:
$$
\mbox{$f\le g$ point-wise in $\R^N$ implies $\mu(f)\le \mu(g)$.}
$$
In this case there exists a non-negative Radon measure $\nu$, whose support is contained in $[\delta,\infty)$ for some $\delta>0$, such that
$$
\mu(f)=\int_0^\infty V_N(\{x\colon f(x)\ge t\})\,d\nu(t),\quad\forall\, f\in\C^N.
$$
In other words, $\mu(f)$ is a weighted mean of the volumes of the super-level sets of $f$.

\medskip

In the last section we introduce the notion of Klain function for valuations on $C^N$. Once more, the idea originates from
a corresponding notion in convex geometry. The Klain function $\KL_\mu$ associated to a continuous, translation 
invariant and $k$-homogeneous valuation $\mu$ on $\K^N$ was introduced in \cite{Klain2}. The definition of $\KL_\mu$ is based on the fact that
the restriction of $\mu$ to convex bodies contained in a $k$-dimensional subspace $E$ of $\R^N$ is a multiple with factor $c_E$ of the $k$-dimensional volume;
this is a consequence of the theorem of Hadwiger mentioned before. $\KL_\mu\colon\Gr(N,k)\,\to\R$ is defined by 
$$
\KL_\mu(E)=c_E,
$$
where $\Gr(N,k)$ is the Grassmannian of the $k$-dimensional subspaces of $\R^N$. 
One of the main results concerning this function is that $\KL_\mu$ determines $\mu$ uniquely if in addition $\mu$ is even. In other words the application
$$
\mu\ \rightarrow\ \KL_\mu
$$
is injective in the space of even, continuous, translation invariant and $k$-homogeneous valuations on $\K^N$ (see \cite{Klain2} and 
\cite[Theorem 6.4.11]{Schneider}).

\medskip

We follow a similar path to define the Klain function of a $k$-homogeneous,
translation invariant and continuous valuation $\mu$ on $\C^N$, $k\in\{0,\dots, N\}$.
We fix $E\in\Gr(N,k)$ and we consider the restriction of $\mu$ to the set of all quasi-concave functions with $\supp(f)\subseteq E$ 
(where $\supp(f)$ is the support of the function $f$, which is a convex subset of $\R^N$). Such a restriction can be identified with a $k$-homogeneous,
translation invariant and continuous valuation on $\C^k$. Applying Theorem $\ref{teorema 2}$,
we deduce that there exists a function $c_E=c_E(t)$, depending on $E$ and defined for $t\geq 0$ such that
\begin{equation}\label{0.8}
\mu(f)=\int_{0}^{+\infty}c_{E}(t)dS_{k}(f;t)
\end{equation}
for all $f\in \C^N$ with support contained in $E$ (the measure $S_k$ is defined in a similar way to $S_N$). 
Hence we define the Klain function of $\mu$ as the function 
$$
\KL_{\mu}\colon(0,+\infty)\times \Gr(N,k)\rightarrow\mathbb{R},\ \KL_{\mu}(t,E)=c_{E}(t).
$$
We will see that this function, as in the case of convex bodies, determine $\mu$ uniquely if additionally $\mu$ is assumed to be even.

\medskip

Finally let us suppose that the valuation $\mu$ is also invariant w.r.t.\ the group $\rot(N)$ of rotations of $\R^N$. 
In this case the Klain function does not depend on $E$: 
$$
\KL_{\mu}(t,E)=\varphi(t)
$$
for some function $\varphi$ of one real variable (depending on $\mu$). With this assumption we are able to prove that 
$$
\mu(f)=\int_{0}^{+\infty}\varphi(t)dS_k(f;t)
$$ 
for all $f\in \C^N$. In this way we retrieve the characterization result proved in 
\cite{Colesanti-Lombardi}.

\bigskip

\noindent
{\bf Acknowledgements.}
The work of the third author was supported by the ETH Zurich Postdoctoral Fellowship Program and the Marie Curie Actions for People COFUND Program. The third author wants to especially thank Michael Eichmair and Horst Kn\"orrer for being his mentors at ETH Zurich.

\section{Background}

\subsection{Notation}

We work in the $N$-dimensional Euclidean space $\mathbb{R}^{N}$, $N\ge1$, endowed with the usual scalar product $(\cdot,\cdot)$ and norm $\|\cdot\|$. 
Given a subset $A$ of $\mathbb{R}^{N}$, $\interno(A)$, $\cl(A)$ and $\partial A$ denote the interior, the closure and the topological boundary 
of $A$, respectively. 

For every $x\in\R^N$ and $r\ge0$, $B_r(x)$ is the closed ball of radius $r$ centered at $x$; for simplicity
we will write $B_r$ instead of $B_r(0)$.  

By a rotation of $\R^N$ we mean a linear transformation of $\R^N$ associated to a matrix in $\rot(N)$. 
A {\em rigid motion} of $\R^N$ is the composition of a translation and a rotation of $\R^N$. 

The Lebesgue measure in $\mathbb{R}^{N}$ will be denoted by $\Vol_{N}$.

Given two functions $f,g\colon\R^N\to\R^N$, we set
$$
(f\wedge g)(x)=\min\{f(x),g(x)\},\quad
(f\vee g)(x)=\max\{f(x), g(x)\},\quad\forall\, x\in\R^N.
$$

\subsection{Convex bodies}

We denote by $\K^N$ the set of all convex bodies, i.e.\ compact and convex subsets of $\R^N$. Our reference book for the theory of convex bodies is the monograph 
$\cite{Schneider}$. 

We define the dimension of $K\in \K^N$ as the minimum natural number $m$ such that there exists an $m$-dimensional affine subspace of 
$\R^N$, containing $K$. 

$\K^N$ is a complete metric space with respect to the Hausdorff distance, defined as follows: 
$$
\delta(K,L)=\max\{\sup_{x\in K}{\rm dist}(x,L),\sup_{y\in L}{\rm dist}(K,y)\},\ \forall K,L\in K^N.
$$
The Minkowski sum of two convex bodies $K$ and $L$ is 
$$
K+L:=\{x+y:x\in K, y\in L \};
$$
the multiplication of a convex body $K$ by a non-negative  real $s$ is:
$$
s\, K:=\{sx\,:\,x\in K\}.
$$

A fundamental tool for the study of convex bodies is the support function. For a fixed convex body $K$, we define 
$h_{K}\colon\sfe\rightarrow \R$ by
$$
h_{K}(u)=\max_{x\in K}(u,x).
$$
The 1-homogeneous extension of the support function of a convex body is a convex function. Vice versa, if $h:\sfe\rightarrow \R^N$ has a 1-homogeneous 
extension on $\R^N$ that is convex, then there exists a unique convex body $K$ such that $h=h_{K}$. In particular, the support function $h_{K}$ identifies $K$.

Two useful properties of support functions are listed below
\begin{itemize} 
\item $K\subseteq L$ if and only if $h_{K}\leq h_{L}$, point-wise on $\sfe$;
\item
$$
h_{sK+tL}=sh_{K}+th_{L},\quad\forall\ K,L\in \K^N,\ \forall\ s,t\ge0.
$$
\end{itemize}

By its convexity, the support function is continuous on $\sfe$; we can evaluate the Hausdorff distance between two convex bodies  
in terms of their support functions. In fact, it holds that
$$
\delta(K,L)=||h_{K}-h_{L}||_{\infty}
$$
for all $K,L\in \K^N$, where $||.||_{\infty}$ is the supremum norm of continuous function on $\sfe$.

Throughout this paper we will frequently use the notion of intrinsic volumes for which we refer to \cite[Chapter 4]{Schneider}. Following the
standard notation, the intrinsic volumes of a convex body $K$ will be denoted by $V_i(K)$, $i=0,\dots,N$. Note that, in particular,
$V_N(K)=\Vol_N(K)$ for every $K\in K^N$.

If $P$ is a subset of $\R^N$, then $P$ is said to be a polytope if it is the convex hull of finitely many points in $\R^N$. 
We will denote by $\Part^N$ the set of all polytopes in $\R^N$. Clearly, $\Part^N\subseteq \K^N$; moreover, the closure of 
$\Part^N$ with respect to the Hausdorff distance is $\K^N$. In other words, for every $K\in \K^N$, there exists a sequence of polytopes $\{P_{j}\}_{j\in \N}$ 
such that $P_{j}\rightarrow K$ with respect to $\delta$, as $j\rightarrow+\infty$ (see for instance \cite[Chapter 1]{Schneider}).

A convex body $Z$ is said to be a zonotope if it is the Minkowski sum of finitely many segments. 
A convex body $Y$ is said to be a zonoid if it is the limit (w.r.t.\ the Hausdorff metric) of a sequence of zonotopes.

Finally, we denote with $\Gr(N,k)$ the Grassmannian of all $k$-dimensional vector subspaces of $\R^N$; this is a compact metric space with respect to the distance
$$
d(E,F)=\delta(E\cap B_1,F\cap B_1),\ \forall E,F\in \Gr(N,k).
$$

\subsection{Characterization results for valuations on $\Val(K^N)$}

In this section we briefly recall some results concerning valuations on convex bodies, that will be used in the sequel.
We begin with some definitions and properties about valuations on $\K^n$.

\begin{defn}
A functional $\mu\colon\K^N \rightarrow \R$ is called a {\em valuation} if
$$\mu(H\cup K)+\mu(H\cap K)=\mu(H)+\mu(K)$$ 
for all $H,K\in \K^N$ such that $H\cup K\in \K^N$.
We also set $\mu(\varnothing)=0$.
\end{defn}

\begin{remark}
\textnormal{We will consider valuations defined not necessarily on $\K^N$, but also on some of its subsets, like $\Part^N$.}

\end{remark}

Valuations that are continuous with respect to the Hausdorff metric and rigid motion invariant are completely characterized
by the following celebrated theorem of Hadwiger (see \cite{Hadwiger}, \cite{Klain-Rota} and \cite{Schneider} for the proof).

\begin{thm}[Hadwiger]\label{Hadwiger}
The functional $\mu\colon\K^N\ \rightarrow\ \R$ is a continuous and rigid motion invariant valuation if and only if there exist $(N+1)$ real coefficients $c_{0},...,c_{N}$ such that
$$
\mu(K)=\sum_{i=0}^{N}c_{i}V_{i}(K)
$$
for all $K\in{\cal K}^{N}$, where $V_{i}$ are the intrinsic volumes.
\end{thm}

Note that, conversely, every linear combination of the intrinsic volumes is a continuous and rigid motion invariant valuation on $\K^N$.

If we add the simplicity hypothesis for a valuation $\mu$, i.e.\ $\mu(K)=0$ if $K$ is a convex body with $\dim(K)<N$, 
then we get the following characterization theorem (see \cite{Klain-Rota} and \cite{Schneider}).

\begin{thm}[Volume Theorem]\label{VT}
A continuous and rigid motion invariant valuation $\mu$ is simple if and only if there exists a real constant $c\in \R$ such that
$$
\mu=c \V_{N}.
$$
\end{thm}

Next, we consider valuations that are continuous and translation invariant; we denote the space of such valuations by $\Val(\K^N)$.

\begin{defn}
Fix $k\in \{0,...,N\}$. A valuation $\mu\in \Val(\K^N)$ is said to be $k$-homogeneous if 
$$
\mu(\lambda K)=\lambda^{k}\mu(K),\quad\forall\, K\in \K^N,\ \forall\, \lambda>0.
$$ 
We will write $\Val_{k}(\K^N)$ to denote the subset of $k$-homogeneous valuations in $\Val(\K^N)$.
\end{defn}

The celebrated McMullen decomposition theorem (see \cite[Chapter 6]{Schneider}) establishes, roughly speaking, that any element of $\Val(\K^n)$ can be decomposed into a
sum of homogeneous valuations of integer degree. We start with the version of this result which concerns polytopes. 

We first recall that a valuation $\mu$ is {\em rational homogeneous} 
of degree $k$ if $\mu(\lambda K)=\lambda^{k}\mu(K)$ for all $K\in \K^N$ and $\lambda>0$ a rational number.

\begin{thm}[McMullen]\label{MP}
If $\mu\colon\Part^N\rightarrow \R$ is a traslation invariant valuation, then there exist $(N+1)$ translation invariant valuations $\mu_{0},...,\mu_{N}$ on $\Part^N$ such that, for 
every $k=0,\dots,N$, $\mu_{k}$ is rational homogeneous of degree $k$ and 
$$
\mu=\sum_{k=0}^{N}\mu_{k}.
$$
\end{thm}

We also have the following theorem.

\begin{thm}[Hadwiger]\label{VTP}
If $\mu\colon\Part^N\rightarrow \R$ is a translation invariant, $N$-homogeneous valuation, then there exists a real constant $c$ such that 
$$
\mu=c \V_{N}.
$$
\end{thm}

If we add the continuity assumption we may pass from $\Part^N$ to $\K^N$.

\begin{thm}[McMullen]
If $\mu\in \Val(\K^N)$, then there exist $\mu_{k}\in \Val_{k}(\K^N)$, $k\in \{0,\dots,N\}$, such that 
$$
\mu=\sum_{k=0}^{N}\mu_{k}.
$$
In particular,
$$
\mu(\lambda K)=\sum_{k=0}^{N}\lambda^{k}\mu_{k}(K)
$$
for all $K\in \K^N$ and for all $\lambda>0$.
\end{thm}

As a consequences of Hadwiger's theorem, we have the following characterization theorem.

\begin{thm}\label{TH}
If $\mu$ is a translation invariant, $N$-homogeneous and continuous valuation on $\K^N$, then there exists a real constant $c$, such that
$$
\mu=c \V_{N}
$$
in $\K^N$.
\end{thm}

We are ready to define the Klain function. Let us fix $k\in\{1,\dots,N-1\}$ and $\mu\in \Val_{k}(\K^N)$. 
We consider $E\in \Gr(N,k)$; let us denote by $\K(E)$ the set of convex bodies contained in $E$. Clearly
$\K(E)$ can be identified with $\K^k$. The restriction of $\mu$ to $\K(E)$, denoted by $\mu|_{E}$, retains the 
properties of continuity, translation invariance and homogeneity of $\mu$. By Theorem \ref{TH}, there exists a constant, that
in general depends on $E$ and will therefore be denoted by $c_E$, such that
\begin{equation}
\mu|_{E}=c_{E} \Vol_{k}
\end{equation}
where $\Vol_{k}$ denotes the $k$-dimensional Lebesgue measure.

The previous equality in fact defines the Klain function $\KL_{\mu}$
$$
\KL_{\mu}\colon\Gr(N,k)\rightarrow \R,\quad
\KL_{\mu}(E)=c_{E},\quad\forall\ E\in\Gr(N,k).
$$
By the continuity of $\mu$, $\KL_{\mu}$ is continuous on $\Gr(N,k)$.

We conclude this section with the notion of an even valuation and a property of the Klain function with this additional hypothesis for the valuation.
\begin{defn}
A valuation $\mu\in \Val_{k}(\K^N)$ is even if $\mu(K)=\mu(-K)$ for all $K\in \K^N$. 
We will write $\Val^{+}_{k}(\K^N)$ to denote the subset of even valuations in $\Val_{k}(\K^N)$.
\end{defn}

\begin{thm}[Klain]\label{TK}
Let $\mu\in \Val^{+}(\K^N)$ be simple. Then there exists a constant $c\in \R$ such that
$$
\mu(K)=c \V_{N}(K),\quad\forall K\in\K^N.
$$
\end{thm}

We will need a variant of the previous result, where continuity is replaced by continuity with respect to monotone sequences.

\begin{thm}\label{TK.2}
Let $\mu$ be a simple, even and translation invariant valuation, continuous with respect to decreasing sequences. Then there exists a constant $c\in \R$  such that
\begin{equation}\label{TK.2 eqn}
\mu(K)=c \V_{N}(K),\quad \forall K\in\K^N.
\end{equation}
\end{thm}

\begin{proof}
We may use the same arguments to prove Klain's Theorem \ref{TK}, presented, for instance, in \cite{Schneider} and \cite{Klain-Rota}, with a slight 
modification. In fact, the only additional thing we need to show is that for every zonoid $Y$ there exists a decreasing sequence of zonotopes converging to $Y$. Indeed, fix a 
natural number $k\geq 1$ and define
$$
Y_{k}=Y+\frac{1}{2^{k}}B_{1}.
$$ 
$Y_{k}$ is again a zonoid and
$$
h_{Y_{k}}=h_{Y}+\frac{1}{2^{k}}.
$$
As every zonoid is the limit of zonotopes, there exists a zonotope $Z_{k}$ such that 
$$
\delta(Z_{k},Y_{k})\leq \frac{1}{4}\,\frac{1}{2^{k}}.
$$ 
We have
$$
\delta(Y,Z_{k})\rightarrow 0\ \textnormal{as}\ k\rightarrow+\infty,
$$
as a consequence of the triangle inequality and
$$
h_{Y}+\frac{3}{2^{k+2}}=h_{Y_{k}}-\frac{1}{2^{k+2}}\leq h_{Z_{k}}\leq h_{Y_{k}}+\frac{1}{2^{k+2}}.
$$
Then we obtain 
$$
Y+\frac{3}{2^{k+2}} B_{1}\subseteq Z_{k}\subseteq Y+\frac{5}{2^{k+2}}B_{1}.
$$
So, we have the following conclusions:
\begin{itemize}
\item $Z_{k}\supseteq Y$;
\item $Z_{k}\supseteq Z_{k+1}$. 
\end{itemize}

Therefore, for every zonoid $Y$, we have created a decreasing sequence of zonotopes, converging to $Y$.
\end{proof}

\begin{thm}\label{TU}
Every valuation $\mu\in \Val^{+}_{k}(\K^N)$ is uniquely determined by its Klain function.
\end{thm}

For the proof, see \cite{Schneider}.

\section{Quasi-concave functions}

\begin{defn} A function $f\colon\R^N\rightarrow \mathbb{R}$ is said to be {\em quasi-concave} if
\begin{itemize}
\item $f(x)\geq 0$ for every $x\in \mathbb{R}^{N}$,
\item for every $t>0$, the set
$$
L_{t}(f)=\{x\in\mathbb{R}^{N}:\ f(x)\geq t\}
$$
is either a convex body or is empty.
\end{itemize}
We will denote by ${\cal C}^N$ the set of all quasi-concave functions.
\end{defn}

The proof of the following result can be found in \cite{Colesanti-Lombardi}.

\begin{lemma}
If $f$ belongs to ${\cal C}^N$, then:
\begin{itemize}
\item $f$ admits a maximum in $\R^N$: 
$$
\sup_{x\in\R^N} f(x)=\max_{x\in \R^N}f(x)<+\infty;
$$
\item $\lim_{||x||\rightarrow+\infty}f(x)=0$;
\item $f$ is upper semi-continuous.
\end{itemize}
\end{lemma}

\begin{example}
If $K\in \K^N$, then we define its characteristic function as
$$
I_{K}:\mathbb{R}^{N}\rightarrow\mathbb{R},\quad 
I_{K}(x)=\begin{cases}
1\; &\mbox{if $x\in K$,}\\
0\; &\mbox{if $x \notin K$.}
\end{cases}
$$
For every $t\ge0$ and $K\in \K^N$, the function $tI_{K}$ belongs to ${\cal C}^N$.
\end{example}

\begin{remark}{\rm
It is easy to show that if $f,g\in {\cal C}^N$, then $f\wedge g \in {\cal C}^N$, while, in general, $f\vee g$ does not belong to $\C^N$.}
\end{remark}

Let $f\in\C^N$; we define the {\em support} of $f$ as the set
$$
cl(\{x\in\R^N\colon f(x)>0\}).
$$
It is easy to verify that the quasi-convexity of $f$ implies that the support of $f$ is convex. In particular, the dimension of the support of $f$ 
is an integer belonging to $\{0,1,\dots,N\}$. The support of $f$ will be denoted by $\supp(f)$. 

\subsection{Simple functions}

Simple functions form a class which will be particularly useful for approximating general quasi-convex functions.

 \begin{defn}
A function $f\colon\R^N\rightarrow\R$ is called simple if it can be written in the form
\begin{equation}\label{simple function}
f=t_1 I_{K_1}\vee...\vee t_m I_{K_m}
\end{equation}
where $0<t_1<\dots<t_m$, and $K_1,\dots,K_m$ are convex bodies such that
$$K_1\supset K_2\supset...\supset K_m.$$
\end{defn}
It is easy to verify that for every $t>0$
\begin{equation}
L_t(f)=\{x\in\R^N\,:\,f(x)\ge t\}=
\left\{
\begin{array}{lllll}
K_i\quad&\mbox{if $t\in(t_{i-1},t_i]$ for some $i=1,\dots m$,}\\
\\
\varnothing\quad&\mbox{if $t>t_m$,}
\end{array}
\right.
\end{equation}
where we have set $t_0=0$. 
Consequently, every simple function belongs to ${\mathcal C}^N$.

\subsection{Valuations on $\C^N$}

\begin{defn}
A functional $\mu\colon{\cal C}^{N}\rightarrow \mathbb{R}$ is said to be a {\em valuation} if
\begin{itemize}
\item $\mu(\underline{0})=0$, where $\underline{0}\in{\cal C}^{N}$ is the function identically equal to zero;
\item for every $f$ and $g\in {\cal C}^{N}$ such that $f\vee g\in{\cal C}^{N}$, we have
$$
\mu(f)+\mu(g)=\mu(f\vee g)+\mu(f\wedge g).
$$
\end{itemize}
\end{defn}

A valuation $\mu$ is said to be rigid motion invariant if for every rigid motion $T\colon\mathbb{R}^{N}\rightarrow \mathbb{R}^{N}$ 
and for every $f\in{\cal C}^{N}$, we have
$$
\mu(f)=\mu(f\circ T)
$$
where $\circ$ denotes the usual composition of functions. We recall that by a rigid motion we mean the composition of a rotation and a 
translation (note the composition between a quasi-concave function and a rigid motion is still a quasi-concave function).

\begin{defn}
A valuation $\mu$ is said to be continuous if for every sequence $\{f_{i}\}_{i\in\mathbb{N}}\subseteq {\cal C}^{N}$ and $f\in{\cal C}^{N}$ such that 
$f_i$ converges point-wise to $f$ in $\R^N$, and $f_i$ is either monotone increasing or decreasing with respect to $i$, pointwise in $\R^N$, we have
$$
\mu(f_{i})\rightarrow\mu(f),\ for\ i\rightarrow +\infty.
$$
\end{defn}

\begin{remark}
\textnormal{If we drop monotonicity in the previous definition, the corresponding notion of continuity, together with translation invariance, becomes too strong. 
Indeed, if $\mu$ is a translation invariant valuation on $\C^N$ such that 
$$
\lim_{i\to\infty}\mu(f_i)=\mu(f)
$$
for every sequence $f_i$, $i\in\N$, in $\C^N$, converging point-wise to $f\in\C^N$,  then $\mu$ is identically zero in $\C^N$. To see this, let $f\in\C^N$ 
have compact support. The sequence $f_i(x)=f(x-i)$, $i\in\N$, converges point-wise to the function identically equal to zero, so that $\mu(f)=0$.
As any function in $\C^N$ is the point-wise limit of functions with compact support, we get $\mu\equiv0$.}
\end{remark}

We can construct some examples of valuations on ${\cal C}^{N}$ using the intrinsic volumes. Let $k\in\{0,\dots,N\}$. For $f\in\C^N$, consider the function
$$
t\,\rightarrow\,u_{k}(t)=V_k(L_t(f))\,\quad
t>0.
$$
This is a decreasing function, which vanishes for $t>M(f)=\max_{\R^n}f$. In particular, $u_{k}$ has bounded variation in
$[\delta,M(f)]$ for every $\delta>0$, hence there exists a Radon measure defined in $(0,\infty)$, which we will denote by 
$S_k(f;\cdot)$, such that 
$$
\mbox{$-S_k(f;\cdot)$ is the distributional derivative of $u_{k}$}.
$$
This means that for every $\varphi\in C_{0}^{\infty}(0,+\infty)$ we have
$$\int_{0}^{+\infty}\varphi'(t)u_{k}(t)dt=\int_{0}^{+\infty}\varphi(t) dS_{k}(f;t).$$

Now, let $\phi$ be a continuous function defined on $[0,\infty)$ such that $\phi(0)=0$. We consider the functional on $\C^N$ defined by
\begin{equation}\label{integral valuation 1}
\mu(f)=\int_{(0,\infty)}\phi(t) dS_k(f;t),\quad f\in\C^N.
\end{equation}
Under suitable assumptions on $\phi$, the functional $\mu$ is a well-defined, continuous and rigid motion invariant valuation. Indeed,
we have the following result (for the proof see \cite{Colesanti-Lombardi}).

\begin{proposition}
Let $\phi\in C([0,\infty))$ and $k\in\{1,\dots,N\}$. Then $\phi$ has finite integral with respect to the measure
$S_k(f;\cdot)$ for every $f\in\C^N$, if and only if $\phi$ satisfies the following condition:
\begin{equation}\label{int cond 2}
\mbox{$\exists\ \delta>0$ s.t. $\phi(t)=0$ for every $t\in[0,\delta]$.}
\end{equation}
Moreover, if $\phi\in C([0,\infty))$ is such that condition \eqref{int cond 2} holds, then the functional \eqref{integral valuation 1} 
defines a rigid motion invariant and continuous valuation on $\C^N$. 
\end{proposition}

\begin{remark}\label{zero omogeneo}
\textnormal{For $k=0$, $V_{0}$ is the Euler characteristic and $S_{0}$ is the Dirac point mass measure concentrated at $M(f)=\max_{\mathbb{R}^{N}}f$. 
So the functional becomes 
$$
\mu(f)=\phi(M(f))$$
and the condition on $\phi$ is not necessary.}
\end{remark}

The following characterization of rigid motion invariant and continuous valuations on quasi-concave functions was proved in \cite{Colesanti-Lombardi}.

\begin{thm}\label{characterization} A functional $\mu$ is a rigid motion invariant and continuous valuation on $\C^N$ if and only if there exist $(N+1)$ continuous 
functions $\phi_k$, $k=0,\dots,N$, defined on $[0,\infty)$ with the following properties: there exists $\delta>0$ such that $\phi_k\equiv0$ in $[0,\delta]$ 
for every $k=1,\dots,N$ and
$$
\mu(f)=\sum_{k=0}^N\int_{[0,\infty)}\phi_k(t)dS_k(f;t),\quad\forall\, f\in\C^N.
$$ 
\end{thm}

\begin{remark}\label{r1}
\textnormal{
Let $\mu$ be a continuous and rigid motion invariant valuation on $\K^N$.
We may apply the previous result to a function of the form $f=sI_{K}$ with $s>0$ and $K\in \K^N$. 
Note that, for $k=0,\dots,N$, it is easy to see that $S_k(f;\cdot)=V_{k}(K)\delta_{s}(\cdot)$, where
$\delta_s$ is the Dirac point-mass measure concentrated at $s$. Therefore
$$
\mu(s I_K)=\sum_{k=0}^N\phi_k(s)V_k(K).
$$}
\end{remark}

\begin{remark}\label{monotone case} 
{\em As proved in \cite[Theorem 1.3]{Colesanti-Lombardi}, if we assume additionally that $\mu$ is monotone (e.g.\ increasing)
in Theorem \ref{characterization}, we obtain also a different representation formula. Namely, there exist $\nu_k$, $k=0,\dots,N$,
non-atomic Radon measures on $[0,\infty)$ such that
$$
\mu(f)=\sum_{k=0}^N\int_{[0,\infty)} u_k(t)\,d\nu_k(t)\quad\forall\, f\in\C^N.
$$
Moreover, there exists a $\delta>0$ such that the support of $\nu_k$ is contained in $[\delta,\infty)$ for every $k\ge1$.} 

\end{remark}

\section{Characterization of simple and $N$-homogeneous valuations}

From now on we consider valuations on $\C^N$ which are continuous and {\em translation invariant} (instead of rigid motion invariant).
We will denote by $\Val({\cal C}^N)$ the space of such valuations.

In the decomposition result that we are going to prove for elements of $\Val(\C^N)$, {\em homogeneous valuations}
will have a central role.  Let $f\in {\cal C}^N$, we fix $\lambda>0$ and define $g_{\lambda}\colon\mathbb{R}^{N}\rightarrow \mathbb{R}^{N}$, as 
$g_{\lambda}(x)=\frac{x}{\lambda}$. It is easy to check that $f\circ g_\lambda$ is still an element of $\C^N$. Moreover, we set
$h\colon\R^N\to\R^N$, $h(x)=-x$. Similarly, $f\circ h\in\C^N$ for every $f\in\C^N$.

\begin{defn}
A valuation $\mu\in \Val({\cal C}^N)$ is said to be $k$-homogeneous, with $k\in \R$, if 
$$
\mu(f\circ g_{\lambda})=\lambda^{k}\mu(f)
$$
for all $f\in {\cal C}^N$ and $\lambda>0$. The class of $k$-homogeneous valuations in $\Val({\cal C}^N)$ will be denoted by $\Val_{k}({\cal C}^N)$.

A valuation $\mu\in \Val({\cal C}^N)$ is said to be even if 
$$
\mu(f\circ h)=\mu(f),
$$ 
for all $f\in {\cal C}^N$. The class of even (resp. $k$-homogenous and even) valuations in $\Val(\C^N)$ will be denoted by
$\Val^{+}({\cal C}^N)$ (resp. $\Val^{+}_k({\cal C}^N)$).
\end{defn}

\begin{remark}\label{l2}
\textnormal{Let $\lambda>0$ and $\mu\in \Val({\cal C}^N)$. The functional $\overline{\mu}:{\cal C}^N\rightarrow\mathbb{R}$ defined by 
$$
\overline{\mu}(f)=\mu(f\circ g_{\lambda}),\quad\forall\, f\in\C^N,
$$ 
belongs to $\Val({\cal C}^N)$.}
\end{remark}

The next proposition gives us a tool that will be useful in the rest of the paper.

\begin{proposition}\label{p1}
Let $\mu\colon{\cal C}^N\rightarrow \mathbb{R}$ be a continuous valuation. If
$$ \mu(tI_{P})=0,\ \forall t>0,\ \forall P\in \Part^N,$$
then $\mu\equiv 0$.
\end{proposition}

The proof of this result follows closely the lines of that of Theorem 1.2 in \cite{Colesanti-Lombardi}.

\begin{proof}

\textbf{First step.} We know, by \cite{Schneider}, that for every $K\in \K^N$ there exists a decreasing sequence of polytopes $P_{i}$, $i\in\N$, 
such that $P_{i}$ converges to $K$. Since the sequence is decreasing, we have
$$
tI_{P_{i}}\searrow tI_{K},\quad \textnormal{as}\quad i\rightarrow +\infty.
$$
By continuity of the valuation, we have $\mu(tI_{P_{i}})\rightarrow\mu(tI_{K})$. Since $\mu(tI_{P_{i}})=0$ for all $i\in \N$, we obtain 
$\mu(tI_{K})=0$ for all $t>0$ and $K\in \K^N$.

\textbf{Second step.} We prove that $\mu$ vanishes on simple functions $g$ of the form 
$$
g=\bigvee_{i=1}^{n}t_{i}I_{K_{i}},
$$
where $n\in\N$ is fixed, $0<t_1<\dots<t_n$, and $K_1,\dots,K_n$ are convex bodies such that
$K_1\supset K_2\supset\dots\supset K_n$. We proceed by induction on $n$.

If $n=1$, the assertion follows from the previous step. Assume, now, that the claim has been proved up to some $n\geq 1$. Let 
$$
g=\bigvee_{i=1}^{n+1}t_{i}I_{K_{i}}
$$ 
be a simple function. Using the valuation property of $\mu$, we can write
\begin{eqnarray*}
\mu(g)&=&\mu\left(\bigvee_{i=1}^{n}t_{i}I_{K_{i}}\right)+\mu(t_{n+1}I_{K_{n+1}})-\mu\left(\bigvee_{i=1}^{n}t_{i}I_{K_{i}}\wedge t_{n+1}I_{K_{n+1}}\right)\\
&=&-\mu\left(\bigvee_{i=1}^{n}t_{i}I_{K_{i}}\wedge t_{n+1}I_{K_{n+1}}\right).
\end{eqnarray*}
The last equality holds as $\mu(t_{n+1}I_{K_{n+1}})=\mu(\bigvee_{i=1}^{n}t_{i}I_{K_{i}})=0$ by the induction assumption. Since 
$$
\bigvee_{i=1}^{n}t_{i}I_{K_{i}}\wedge t_{n+1}I_{K_{n+1}}=t_{1}I_{K_{n+1}},
$$ 
we obtain $\mu(g)=0$ again by the induction assumption.

\textbf{Third step.} We consider now a quasi-concave function $f$ that has bounded support. That means there exists a convex body $K$ such that 
$f(x)=0$ for every $x\notin K$ or, equivalently, $L_{t}(f)\subset K$ for all $t>0$. Let $M(f)=\max_{\mathbb{R}^{N}}f$ and define, for $i\in \mathbb{N}$, 
$$
P_{i}=\left\{t_{j}=j\frac{M(f)}{2^{i}}:j=1,...,2^{i}\right\},$$
and 
$$
f_{i}=\bigvee_{j=1}^{2^{i}}t_{j}I_{L_{t_{j}}(f)}.
$$
The sequence $f_{i}$, $i\in\N$, is an increasing sequence of simple functions that converges point-wise to $f$ in ${\cal C}^N$, as $i\rightarrow+\infty$. 
So, by the continuity of $\mu$, we have $\mu(f_{i})\rightarrow\mu(f)$, as $i\rightarrow+\infty$. Since every $f_{i}$ is simple, the previous step implies 
$\mu(f_{i})=0$. Hence $\mu(f)=0$.

\textbf{Fourth step.} Finally, we consider an arbitrary $f\in{\cal C}^N$ and, for $i\in \mathbb{N}$, we define $B_{i}=\{x\in\mathbb{R}^{N}: ||x||\leq i \}$ and $f_{i}=fI_{B_{i}}$.
Then we have an increasing sequence of quasi-concave functions with compact support, that converges point-wise to $f$, so that $\mu(f_{i})\rightarrow\mu(f)$. 
We get $\mu(f)=0$, since $\mu(f_{i})=0$ for every $i$ by the previous part of the proof.
\end{proof}

\begin{remark}\label{connection}
\textnormal{The following construction connects valuations on $\C^N$ and valuations on convex bodies.  
Let $\mu\in \Val({\cal C}^N)$, and fix $t>0$; we define
$$
\widetilde{\mu}_{t}\colon\K^N\rightarrow \R\quad \textnormal{by}\quad \widetilde{\mu}_{t}(K)=\mu(tI_{K}),\quad\forall\, K\in\K^N.
$$
From the valuation property and the translation invariance of $\mu$, it follows that this is a translation invariant valuation on $\K^N$.
Moreover, by the continuity of $\mu$, it is continuous with respect to decreasing sequences.} 
\end{remark}

\subsection{Proof of Theorems \ref{teorema 1} and \ref{teorema 2}}

The following is an analogue of McMullen's theorem in ${\cal C}^N$.

\begin{thm}\label{t1}
Let $\mu\colon{\cal C}^N\rightarrow \mathbb{R}$ be a continuous and translation invariant valuation.
There exist a unique $\mu_{k}\in \Val_{k}({\cal C}^N)$, $k=0,\dots,N$, such that
$$
\mu=\sum_{k=0}^{N}\mu_{k}.
$$
\end{thm}

\begin{proof}
If we fix $t>0$, then we can define $\widetilde{\mu}_{t}\colon\Part^N\rightarrow\mathbb{R}$ as 
$$
\widetilde{\mu}_{t}(P)=\mu(tI_{P})\quad\forall\, P\in{\mathcal P}^N,
$$
which is a translation invariant valuation on $\Part^N$. By Theorem \ref{MP}, there exist $\widetilde{\mu}_{t,k}$, $k=0,...,N$, translation invariant and rational $k$-homogeneous 
valuations on $\Part^N$, such that
\begin{equation}\label{aggiunta 1}
\widetilde{\mu}_{t}(\lambda P)=\sum_{k=0}^{N}\lambda^{k}\widetilde{\mu}_{t,k}(P)
\end{equation}
for every $P\in \Part^N$ and for all rational $\lambda>0$.

If we write \eqref{aggiunta 1} for $\lambda=1$, we have 
$$
\mu(tI_{P})=\widetilde{\mu}_{t,0}(P)+...+\widetilde{\mu}_{t,N}(P).
$$
Similarly, for $\lambda=2$, we get
$$
\mu(tI_{2P})=\widetilde{\mu}_{t,0}(P)+...+2^{N}\widetilde{\mu}_{t,N}(P).
$$
We can repeat this argument for $\lambda=3,\dots,N+1$ to obtain a system of $N+1$ linear equations and a matrix of Vandermonde type associated to it:
\begin{equation}
M:=
\begin{pmatrix}
1 & 1 & \dots & 1\\
1 & 2 & \dots & 2^{N}\\
\vdots & \vdots & \ddots & \vdots\\
1 & N+1 & .... & (N+1)^{N}
\end{pmatrix}
\end{equation}
Since $M$ is invertible, we have
\begin{equation}\label{aggiunta1}
\widetilde{\mu}_{t,k}(P)=\sum_{j=1}^{N+1}c_{k,j}\mu(tI_{jP}),
\end{equation}
where $c_{k,j}$ are the coefficients of the inverse matrix of $M$, which are independent of $t$ and of $P$. Moreover, we observe that the continuity of $\mu$ implies that $\widetilde{\mu}_{t}$ is dilation continuous, i.e.\ $\lambda \mapsto \widetilde{\mu}_{t}(\lambda P)$ is continuous for every $P$.

\noindent
Then equation $(\ref{aggiunta1})$ implies that all $\widetilde{\mu}_{t,k}$ are dilation continuous, for all $t>0$. Since $\widetilde{\mu}_{t,k}$ are also rational $k$-homogeneous, we can conclude that they are $k$-homogeneous for all real positive number $\lambda$ and for all $t>0$.

Now, we want to determine a set of valuations $\mu_{k}$ on ${\cal C}^N$, $k=0,\dots,N$, such that $\mu_{k}(tI_{P})=\widetilde{\mu}_{t,k}(P)$ for every polytope $P$ and 
every $t>0$. So we define $\mu_{k}\colon{\cal C}^N\rightarrow \mathbb{R}$ by
\begin{equation}\label{aggiunta2}
\mu_{k}(f)=\sum_{j=1}^{N+1}c_{k,j}\mu(f\circ g_{j}),
\end{equation}
where we recall that for any $s>0$, $g_s\colon\R^N\to\R^N$ is defined by
$$
g_s(x)=\frac xs.
$$
For every $k$, $\mu_{k}$ is a continuous translation invariant valuation on ${\cal C}^N$.

As next step, for $\lambda>0$ and $k=0,\dots,N$, we define $\overline{\mu}_{k}\colon{\cal C}^N\rightarrow \mathbb{R}$ by
$$
\overline{\mu}_{k}(f)=\mu_{k}(f\circ g_{\lambda})-\lambda^{k}\mu_{k}(f),
$$
which turns out to be a continuous valuation. Furthermore, for $P\in \Part^N$, 
$$
\overline{\mu}_{k}(tI_{P})=\mu_{k}(tI_{\lambda P})-\lambda^{k}\mu_{k}(tI_{P})=\widetilde{\mu}_{t,k}(\lambda P)-\lambda^{k}\widetilde{\mu}_{t,k}(P).
$$
Since $\widetilde{\mu}_{t,k}$ is a homogeneous valuation of degree $k$ on $\Part^N$, by McMullen's Theorem \ref{MP}, we have $\overline{\mu}_{k}(tI_{P})=0$.
By Lemma $\ref{p1}$, we obtain $\overline{\mu}_{k}(f)=0$ for all $f\in {\cal C}^N$ and this means that $\mu_{k}$ is homogeneous of degree $k$.

Finally, we define $\widetilde{\mu}\colon{\cal C}^N\rightarrow \mathbb{R}$, as
$$
\widetilde{\mu}(f)=\mu(f)-\sum_{k=0}^{N}\mu_{k}(f),
$$
which results a translation invariant, continuous valuation such that, for every $t>0$ and $P\in \Part^N$, $\widetilde{\mu}(tI_{P})=0$.
Then $\widetilde{\mu}(f)=0$ for every $f\in {\cal C}^N$ by Lemma $\ref{p1}$ and we have the conclusion
$$
\mu(f)=\sum_{k=0}^{N}\mu_{k}(f)
$$
for every $f\in {\cal C}^N$.

\noindent
Finally, we conclude the proof with a remark about uniqueness. If we have 
$$\mu=\sum_{k=0}^{N}\mu_{k}=\sum_{k=0}^{N}\sigma_{k},$$ then we are able to write
$$0=\sum_{k=0}^{N}\mu_{k}-\sigma_{k}.$$
By the homogeneity of $\mu_{k}$ and $\sigma_{k}$, we obtain the uniqueness.

\end{proof}

We conclude this section with two results that are analogous to Theorem $\ref{TK.2}$ and $\ref{TH}$, in the case of valuations on $\C^N$.
We say that a valuation $\mu$ on $\C^N$ is {\em simple}, if $\mu(f)=0$ for all $f\in {\cal C}^N$ such that $\dim(\supp(f))<N$, where we recall that
$\supp(f)$ is the support of $f$, i.e. the set of points where $f$ is strictly positive.

\begin{thm}\label{aggiunta 2}
Let $\mu\in \Val^{+}({\cal C}^N)$; $\mu$ is simple if and only if there exists a unique function $c$ from $[0,+\infty)$ to $\R$ and $\delta>0$, 
with the following properties:
\begin{itemize}
\item $c$ is continuous in $[0,\infty)$,
\item $c(t)=0$ for all $t\in [0,\delta]$,
\end{itemize}
such that
$$
\mu(f)=\int_{(0,+\infty)}c(t)dS_{N}(f;t)
$$
for every $f\in {\cal C}^N$.
\end{thm}

\begin{proof}
For a fixed $t>0$, we consider the map $\widetilde{\mu}_{t}\colon\K^N\rightarrow \R$ defined by 
$$
\widetilde{\mu}_t(K)=\mu(t\,I_K),\quad\forall\,K\in\K^N.
$$
This is a translation invariant, even, valuation, which is additionally simple (as $\mu$ is simple) and continuous with respect to decreasing sequences. Therefore, 
by Theorem $\ref{TK.2}$, there exists a real constant, depending on $t$, $c(t)$, such that 
\begin{equation}\label{aggiunta3}
\widetilde{\mu}_{t}(K)=c(t)\V_N(K),\quad\forall\, K\in\K^N.
\end{equation}
The function $c$ is continuous by the continuity of $\mu$ and it is univocally determined by 
$(\ref{aggiunta3})$. The integral form of $\mu$ and the additional condition on $c$ can be obtained using the same argument of  the proof of 
Lemma 6.5 and Theorem 1.2 of \cite{Colesanti-Lombardi}, that we briefly outline.
We first consider a general simple function $f$ of the form
$$
f=t_1 I_{K_1}\vee...\vee t_m I_{K_m},
$$ 
with $0<t_1<\dots<t_m$, $K_1,\dots,K_m$ convex bodies such that
$$
K_1\supset K_2\supset...\supset K_m.
$$
The extension of \eqref{aggiunta3} to $f$ is
$$
\mu(f)=\sum_{i=1}^{m-1}c(t_{i})(\V_{N}(K_{i})- \V_{N}(K_{i+1}))+c(t_{m})\V_{N}(K_{m}).
$$
This follows easily the valuation property of $\mu$ and  \eqref{aggiunta3}. Then, with a similar argument of Step $3$ in the proof of Proposition \ref{p1}, 
we prove the validity of the integral form of $\mu$ for all $f\in \C^N$ with bounded support.

Using the previous steps we prove that $c$ must vanish in a right neighborhood of the origin. This is the most delicate part of the proof, for 
which we refer the reader to the proof of Lemma 6.5 in \cite{Colesanti-Lombardi}.

Finally, we obtain the validity of the integral representation formula in the general case, by an approximation argument.
\end{proof}

\begin{thm}\label{TH2}
Let $\mu\in \Val_{N}({\cal C}^N)$. There exists a unique function $c\colon[0,+\infty)\rightarrow \R$ and a number $\delta>0$, with the following properties:
\begin{itemize}
\item $c$ is continuous,
\item $c(t)=0$, for all $t\in [0,\delta]$,
\end{itemize}
such that 
$$
\mu(f)=\int_{(0,+\infty)}c(t)dS_{N}(f;t)
$$
for all $f\in {\cal C}^N$.
\end{thm}

\begin{proof}
The same argument as in the proof of Theorem \ref{aggiunta 2} can be used, where Theorem \ref{TK.2} has to be replaced by Theorem \ref{TH}.
\end{proof}

\noindent
We conclude this section with a remark about $0$-homogeneous valuations $\mu \in \Val_{0}(\C^N)$, which we are able to characterize them.

In this case, we have $$\mu(f \circ g_{\lambda})=\mu(f)$$ for all $f\in \C^N$ and $\lambda>0$.
As in the $N$-homogeneous case, we first look to the valuation defined on $\Part^N$. We set, for $t>0$,
$$\tilde{\mu}_{t}\colon \Part^N\rightarrow \R,\ \tilde{\mu}_{t}(P)=\mu(tI_{P}).$$
We observe that $\tilde{\mu}_{t}$ is a $0$-homogeneous, translation invariant valuation, then (by \cite[p. 353]{Schneider}) we are able to say that $\tilde{\mu}_{t}$ is constant on $\Part^N$, 
$$
\tilde{\mu}_{t}(P)\equiv \tilde{\mu}_{t}(\lbrace 0\rbrace)
$$ 
for all $P\in \Part^N$.
So we obtain $\mu(tI_{P})=\mu(tI_{\lbrace0\rbrace})$, i.e.\ it is equal to a function depending only on $t$:
$$
\mu(tI_{P})=\phi(t)
$$ 
for all $t>0$ and $P\in \Part^N$.
In particular, we have that $\phi\colon \R\rightarrow \R$ is continuous, and 
$$
\mu(tI_{P})=\phi(M(tI_{P}))
$$ where $M(tI_{P})=\max_{\R^N}tI_{P}$.

We define now $\overline{\mu}(f)=\phi(M(f))$ for every $f\in\C^N$. This is a continuous, translation invariant valuation on $\C^N$ 
(see Remark \ref{zero omogeneo}). Applying Proposition $\ref{p1}$ to
$\tilde{\mu}:=\mu-\overline{\mu}$, we have $\tilde{\mu}\equiv 0$, hence $\mu(f)=\phi((f))$.

\section{The Klain function on $\Val({\cal C}^N)$}

Let $\mu\in \Val_{k}({\cal C}^N)$ with $k\in\lbrace1,\dots,N-1\rbrace$. 
We fix $E\in \Gr(N,k)$ and we define 
$$
\C(E)=\{f\in \C^N:\supp(f)\subset E \}.
$$ 
After fixing a coordinate system on $E$ we can identify $\C(E)$ as $\C^k$, so we have $\mu|E\in \Val_{k}(\C^k)$.

Applying Theorem $\ref{TH2}$ we deduce that there exists a function $c_E\colon(0,+\infty)\rightarrow \R$, depending on $E$, such that
$$
\mu(f)=\int_{0}^{+\infty}c_{E}(t)dS_{k}(f;t)
$$
for every $f\in H$. Moreover $c_E$ is continuous and there exists $\delta>0$ such that $c_E(t)=0$ for all $t\in [0,\delta]$.

Hence we can define a function, $\KL_{\mu}$, that we will call the {\em Klain function of $\mu$}, as follows:
$$
\KL_{\mu}\colon(0,+\infty)\times \Gr(N,k)\rightarrow\mathbb{R},\ \KL_{\mu}(t,E)=c_{E}(t).
$$
This is equivalent to the identity:
\begin{equation}\label{e1}
\mu(f)=\int_{0}^{+\infty}\KL_{\mu}(t,E)dS_{k}(f;t)
\end{equation}
for all $f\in \C^N$ such that $\supp(f)\subset E$.

We choose now $f=tI_{K}\in \C^N$, where $t>0$ and $K$ is a convex body contained in $E$.
By Remark $\ref{r1}$ we obtain
$$
\mu(tI_{K})=\KL_{\mu}(t,E)\ \Vol_{k}(K).
$$
We have proved the following proposition.

\begin{proposition}
If $\mu\in \Val_{k}(\C^N)$, then for all $E\in \Gr(N,k)$ and $t>0$, there exists a unique real number $c_{t,E}$ such that
$$\widetilde{\mu}_{t}|E=c_{t,E}\Vol_{k}.$$
In particular, $c_{t,E}$ is the Klain function of $\mu$ evaluated at $(t,E)$.
\end{proposition}

\begin{proposition}\label{p3}
The map $\KL\colon\mu\mapsto \KL_{\mu}$ is injective on $\Val^{+}_{k}({\cal C}^N)$.
\end{proposition}

\begin{proof}
Let $\mu,\sigma\in \Val^{+}_{k}({\cal C}^N)$ such that $\KL_{\mu}=\KL_{\sigma}$ in $(0,+\infty)\times \Gr(N,k)$.
We fix $t>0$ and we consider $\widetilde{\mu}_{t}$ and $\widetilde{\sigma}_{t}$ on $\Part^N$. By Theorem $\ref{TU}$, we 
obtain $\widetilde{\mu}_{t}=\widetilde{\sigma}_{t}$ for all $t>0$.

We define now $\overline{\mu}\colon{\cal C}^N\rightarrow \R$ as $\overline{\mu}=\mu-\sigma$, and we observe that 
$\overline{\mu}\in \Val^{+}_{k}({\cal C}^N)$. Furthermore, for $f=tI_{P}$, we have 
$\overline{\mu}(tI_{P})=\mu(tI_{P})-\sigma(tI_{P})=\widetilde{\mu}_{t}(P)-\widetilde{\sigma}_{t}(P)=0.$
Then, by Proposition $\ref{p1}$, we have $\overline{\mu}=0$ in ${\cal C}^N$.
\end{proof}

Let $\mu\in \Val_{k}({\cal C}^N)$ be also rotation invariant, so $\KL_{\mu}(t,E)$ does not depend on $E$.
This follows immediately from the definition of the Klain function and the fact  
that for every $E,F \in \Gr(N,k)$ there exists a rotation $\phi\in \rot(N)$ such that $E=\phi F$.

In this case, we get the Klain function as a function of $t$, hence
$$
\KL_{\mu}(t,E)=\varphi(t)
$$
for $t\in (0,+\infty)$ and all $E\in \Gr(N,k)$.

We have the following characterization theorem.

\begin{thm}\label{t3}
If $\mu\in \Val_{k}({\cal C}^N)$ and it is also rotation invariant, then 
$$\mu(f)=\int_{0}^{+\infty}\varphi(t)dS_k(f;t),$$
where $\varphi(t)=\KL_{\mu}(t,E)$.
\end{thm}

\begin{proof}
We observe that, by the property of the Klain function, the quantity 
$$\int_{0}^{+\infty}\varphi(t)dS_k(f;t)$$ is a continuous, rigid motion and $k$-homogeneous valuation on 
$\C^N$.

\noindent
We conclude the proof applying Proposition $\ref{p3}$ to 
$\mu(\cdot)-\int_{0}^{+\infty}\varphi(t)dS_k(\cdot\ ;t)$.
\end{proof}

\bigskip

A. Colesanti and N. Lombardi:\\ 
Dipartimento di Matematica e Informatica ``U.Dini'', \\
Viale Morgagni 67/A, 50134, Firenze, Italy

\vspace{0.25cm}

L. Parapatits:\\
Vienna University of Technology, Institute of Discrete Mathematics and Geometry\\
Wiedner Hauptstra\ss e 8--10/104, 1040 Wien, Austria

\vspace{0.5cm}

Electronic mail addresses:\\
andrea.colesanti@unifi.it\\
nico.lombardi@unifi.it\\
lukas.parapatits@alumni.tuwien.ac.at

\end{document}